\documentclass[a4paper, oneside,12pt]{amsart}
\usepackage{graphicx}
\usepackage[english]{babel}
\usepackage{mathrsfs,amssymb}
\usepackage{mathtools}
\usepackage[colorlinks, citecolor = blue]{hyperref}
\usepackage{tcolorbox}
\usepackage[shortlabels]{enumitem}
\setlist[itemize]{leftmargin=25pt}
\setlist[enumerate]{leftmargin=25pt}
\usepackage{pgfplots}
\pgfplotsset{every axis/.append style={tick label style={/pgf/number format/fixed},font=\scriptsize,ylabel near ticks,xlabel near ticks,grid=major}}

\usepackage{geometry}
\makeatletter
\newcommand{\leqnomode}{\tagsleft@true}
\newcommand{\reqnomode}{\tagsleft@false}
\makeatother

\usepackage{todonotes}

\newtheorem{theorem}{Theorem}[section]
\newtheorem{theoremLetter}{Theorem}

\newtheorem{lemma}[theorem]{Lemma}
\newtheorem{prop}[theorem]{Proposition}

\theoremstyle{definition}

\theoremstyle{remark}
\newtheorem{remark}[theorem]{Remark}

\numberwithin{equation}{section}

\DeclareMathOperator*{\esssup}{ess\,sup}
\DeclareMathOperator*{\essinf}{ess\,inf}

\let \la=\lambda
\let \e=\varepsilon

\let \o=\omega
\let \a=\alpha
\let \f=\varphi

\let \si=\sigma

\let \ga=\gamma

\def\avint_#1{\mathchoice{\mathop{\kern 0.2em\vrule width 0.6em height 0.69678ex depth -0.58065ex \kern -0.8em \intop}\nolimits_{\kern -0.4em#1}}{\mathop{\kern 0.1em\vrule width 0.5em height 0.69678ex depth -0.60387ex \kern -0.6em \intop}\nolimits_{#1}} {\mathop{\kern 0.1em\vrule width 0.5em height 0.69678ex depth -0.60387ex \kern -0.6em \intop}\nolimits_{#1}} {\mathop{\kern 0.1em\vrule width 0.5em height 0.69678ex depth -0.60387ex \kern -0.6em \intop}\nolimits_{#1}}}

\allowdisplaybreaks

\begin{document}
\title[Weighted weak type inequalities]
{On some improved weighted weak type inequalities}

\author[A.K. Lerner]{Andrei K. Lerner}
\address[A.K. Lerner]{Department of Mathematics,
Bar-Ilan University, 5290002 Ramat Gan, Israel}
\email{lernera@math.biu.ac.il}

\author[K. Li]{Kangwei Li}
\address[K. Li]{Center for Applied Mathematics, Tianjin University, Weijin Road 92, 300072 Tianjin, China}
\email{kli@tju.edu.cn}

\author[S. Ombrosi]{Sheldy Ombrosi}
\address[S. Ombrosi]{Departamento de Análisis Matemático y Matemática Aplicada\\ Universidad Complutense (Spain) \&
Departamento de Matemática e Instituto de Matemática. Universidad Nacional del Sur - CONICET Argentina}
\email{sombrosi@ucm.es}

\author[I.P. Rivera-R\'ios]{Israel P. Rivera-R\'ios}
\address[I.P. Rivera-R\'ios]{Departamento de An\'alisis Matem\'atico, Estad\'istica e Investigaci\'on Operativa
y Matem\'atica Aplicada. Facultad de Ciencias. Universidad de M\'alaga (M\'alaga, Spain).}
\email{israelpriverarios@uma.es}

\thanks{The first author was supported by ISF grant no. 1035/21. The second author was supported by the National Natural Science Foundation of China through project numbers 12222114 and 12001400.
The third author was supported by Spanish Ministerio de Ciencia e Innovaci\'on grant PID2020-113048GB-I00.
The fourth author was supported by Spanish Ministerio de Ciencia e Innovaci\'on grant PID2022-136619NB-I00 and by Junta de Andaluc\'ia grant FQM-354.}

\begin{abstract}
In this paper we obtain the sharp quantitative matrix weighted weak type bounds for the Christ--Goldberg maximal operator $M_{W,p}$ in the case $1<p<2$, improving a recent
result by Cruz-Uribe and Sweeting \cite{CUS23}. Also, in the scalar setting, we improve a weak type bound obtained in \cite{CUS23} for Calder\'on--Zygmund operators.
\end{abstract}


\maketitle
\section{Introduction}
In this paper we consider weighted weak type inequalities of the form
\begin{equation}\label{sw}
|\{x\in {\mathbb R}^d:|w(x)^{1/p}T(fw^{-1/p})(x)|>\a\}|\le\frac{C}{\a^p}\int_{{\mathbb R}^d}|f|^pdx,
\end{equation}
both in the scalar and matrix settings. Here $1\le p<\infty$ and $T$ is a given operator.

Suppose first that $w$ is a scalar weight, that is, $w$ is a non-negative locally integrable function on ${\mathbb R}^d$.
In this case inequalities (\ref{sw}) were first considered by Muckenhoupt and Wheeden in \cite{MW77}, and then studied by many authors, see, e.g., \cite{CUMP05, LOP19, S85}.

Observe that there are more standard weighted weak type inequalities of the form
\begin{equation}\label{nw}
w\{x\in {\mathbb R}^d:|T(f)(x)|>\a\}|\le\frac{C}{\a^p}\int_{{\mathbb R}^d}|f|^pwdx,
\end{equation}
where $w(E):=\int_Ewdx$ for a measurable set $E\subset {\mathbb R}^d$. Even though inequalities of the form (\ref{sw}) are interesting in their own right, they are relatively exotic
compared to (\ref{nw}).

The situation is different in the matrix weight setting. Suppose that $w=W$ is a matrix weight, that is, $W$ is an $n\times n$ self-adjoint matrix function with locally
integrable entries such that $W(x)$ is positive definite for a.e. $x\in {\mathbb R}^d$. For $f: {\mathbb R}^d\to {\mathbb C}^n$ and a linear operator~$T$, define $T(f)$ componentwise.
Then the strong $L^p(W)$ boundedness of~$T$ means that
$$\int_{{\mathbb R}^d}|W(x)^{1/p}T(fW^{-1/p})(x)|^pdx\le C\int_{{\mathbb R}^d}|f|^pdx,$$
and we see that (\ref{sw}) is its natural weak type counterpart, while (\ref{nw}) is meaningless in the matrix setting.

In what follows, we assume that $T$ is a Dini-continuous Calder\'on--Zygmund operator. Given a matrix weight $W$ and $1\le p<\infty$, define
$$T_{W,p}f(x):=W(x)^{1/p}T(fW^{-1/p})(x).$$
We also consider the Christ--Goldberg maximal operator defined by
$$M_{W,p}f(x):=\sup_{Q\ni x}\frac{1}{|Q|}\int_Q|W^{1/p}(x)W^{-1/p}(y)f(y)|dy.$$
For $p=2$ this operator was defined by Christ--Goldberg \cite{CG01} and for $p>1$ by Goldberg~\cite{G03}.

Quantitative matrix weighted inequalities of the form (\ref{sw}) were first considered by Cruz-Uribe et al. \cite{CUIMPRR21} in the case $p=1$.
In a very recent work by Cruz-Uribe and Sweeting \cite{CUS23}, the results of \cite{CUIMPRR21} have been extended to the case $p>1$.
Both results in \cite{CUIMPRR21, CUS23} can be formulated as follows.

\begin{theoremLetter}[\cite{CUIMPRR21, CUS23}]\label{thm:A} Let $1\le p<\infty$. Then
\begin{equation}\label{weakmatrix}
\|T_{W,p}f\|_{L^{p,\infty}}\lesssim [W]_{A_p}^{1+\frac{1}{p}}\|f\|_{L^p},
\end{equation}
and the same bound holds for $M_{W,p}$.
\end{theoremLetter}

In \cite{LLOR23}, the authors showed that in the case $p=1$ the quadratic dependence on $[W]_{A_1}$ in (\ref{weakmatrix}) is best possible both for $T_{W,1}$ and $M_{W,1}$.

Suppose now that $p>1$. Consider Theorem \ref{thm:A} for $M_{W,p}$ in the matrix case and for $T_{w,p}$ in the scalar case.
From the known strong type bound for $M_{W,p}$ (which is due to Buckley~\cite{B93} in the scalar case, and Isralowitz--Moen \cite{IM19} in the matrix case)
one can conclude that
\begin{equation}\label{trwt}
\|M_{W,p}f\|_{L^{p,\infty}}\lesssim [W]_{A_p}^{\frac{1}{p-1}}\|f\|_{L^p}\quad (1<p<\infty).
\end{equation}
Also, by the $A_2$ theorem of Hyt\"onen \cite{H12}, the same bound holds for $T_{w,p}$ when $1<p\le 2$. Therefore, in both cases Theorem \ref{thm:A}
provides a new bound for $p$ satisfying $1+\frac{1}{p}<\frac{1}{p-1}$, namely, for $1<p<\frac{1+\sqrt{5}}{2}$.
It was conjectured in \cite{CUS23} that in this range (\ref{weakmatrix}) is sharp in both cases considered above.

We will show that this conjecture is not true. For $M_{W,p}$ we obtain the sharp weak $L^p$ bound for all $1<p<2$. For $T_{w,p}$ we also obtain an improvement of
Theorem \ref{thm:A} but our new bound is probably not optimal. In order to state the sharpness part of our results, let us define in the scalar setting
$$\f_{M_{p}}(t):=\sup_{[w]_{A_p}\le t}\|M_{w,p}\|_{L^p\to L^{p, \infty}}\quad(t\ge 1).$$
In the similar way define $\f_{T_{p}}(t)$. Our results read as follows.

\begin{theorem}\label{swtb} For all $1<p<2$,
$$
\|M_{W,p}f\|_{L^{p,\infty}}\lesssim [W]_{A_p}^{\frac{2}{p}}\|f\|_{L^p},
$$
and, moreover, this bound is sharp in the sense that $\f_{M_{p}}(t)\gtrsim t^{\frac{2}{p}}$ for all $t\ge 1$.
\end{theorem}

\begin{theorem}\label{czop}
For all $1<p<2$,
\begin{equation}\label{czrs}
\|T_{w,p}f\|_{L^{p,\infty}}\lesssim [w]_{A_p}^{1+\frac{1}{p^2}}\big(\log([w]_{A_p}+e)\big)^{\frac{1}{p}}\|f\|_{L^p}.
\end{equation}
\end{theorem}

Some comments about these results are in order.
\begin{itemize}
\item The example used to prove the sharpness part of Theorem \ref{swtb} is a modification of the corresponding example constructed in the case $p=1$ by the authors in \cite{LLOR23}.
\item We emphasize that Theorem \ref{czop} is scalar, and it is not clear to us whether it can be extended to the matrix case.
Observe that the currently best known strong type bound for $T_{W,p}$ in the matrix case is
$$\|T_{W,p}\|_{L^p\to L^p}\lesssim [W]_{A_p}^{1+\frac{1}{p(p-1)}}\quad(1<p<\infty)$$
(see \cite{NPTV17} for the case $p=2$ and \cite{CIM18} for all $p>1$). It is a big open question whether this bound can be improved to $\max(\frac{1}{p-1},1)$ as in the scalar case.
The weak type bound for $T_{W,p}$ obtained in Theorem \ref{thm:A} provides a new result when $1+\frac{1}{p}<1+\frac{1}{p(p-1)}$, that is, when $1<p<2$, and it is an important question whether
in this range the exponent $1+\frac{1}{p}$ is sharp. Theorem \ref{czop} shows that this is not true in the {\it scalar} case.
\item Even though we stated Theorem \ref{czop} for all $1<p<2$, it provides a new bound in a smaller range of $p$ satisfying $1+\frac{1}{p^2}<\frac{1}{p-1}$.
\item The same example as in Theorem \ref{swtb} shows also that $\f_{H_{p}}(t)\gtrsim~t^{\frac{2}{p}}$ for the Hilbert transform $H$.
For this reason it is tempting to conjecture that the right-hand side of~(\ref{czrs}) can be improved to $[w]_{A_p}^{\frac{2}{p}}$ for all $1<p<2$.
\item In a very recent paper \cite{NSS24}, the authors obtained a multilinear version of Theorem \ref{thm:A} for Calder\'on--Zygmund operators in the scalar setting. 
In the linear case they recover the exponent $1+\frac{1}{p}$ of $[w]_{A_p}$. Therefore it would be interesting to check whether the approach used in the proof of Theorem \ref{czop} 
can be extended to the multilinear setting. 
\end{itemize}

We complement Theorems \ref{swtb} and \ref{czop} by considering the corresponding weak type bounds in the case $p\ge 2$. First consider  $M_{W,p}$.
Comparing the bounds in Theorem \ref{swtb} and in (\ref{trwt}), we see that $\frac{2}{p}<\frac{1}{p-1}$ precisely when $1<p<2$. Therefore, it is natural to conjecture that (\ref{trwt}) is sharp for $p\ge 2$.
The sharpness of the strong type bound follows easily by the standard power weight example. However, the weak type case is more
complicated. While we are not able to establish this conjecture, we obtain a close result by showing that (\ref{trwt}) is `almost' sharp for $p\ge 2$.

\begin{theorem}\label{as}
Let $p\ge 2$. Then, for all $t\ge 1$,
$$
\f_{M_{p}}(t)\gtrsim t^{\frac{1}{p-1}} \big(\log (t+e) \big)^{-\frac 1p}.
$$
\end{theorem}

In particular, this result shows that the exponent $\frac{1}{p-1}$ in (\ref{trwt}) cannot be decreased when $p\ge 2$.

Consider now $T_{w,p}$. By the $A_2$ theorem \cite{H12},
$$\|T_{w,p}\|_{L^{p,\infty}}\lesssim [w]_{A_p}\|f\|_{L_p}\quad(p\ge 2).$$
We will show that this bound is best possible.

\begin{theorem}\label{czsh} Let $p\ge 2$, and let $H$ be the Hilbert transform. Then, for all $t\ge 1$,
$$\f_{H_{p}}(t)\gtrsim t.$$
\end{theorem}

Observe that this result is much simpler than Theorem \ref{as}, namely, the example showing the sharpness here is much more elementary compared to the example used in the proof of Theorem \ref{as}.

The paper is organized as follows. Section 2 contains some preliminary facts. In Section~3 we prove Theorems \ref{swtb} and \ref{as}, and in Section 4 we prove Theorems \ref{czop} and~\ref{czsh}.

Throughout the paper we use the notation $A\lesssim B$ if $A\le CB$ with some independent constant $C$. We write $A\simeq B$ if $A\lesssim B$ and $B\lesssim A$.

\section{Preliminaries}
\subsection{Dyadic lattices, sparse families and Calder\'on--Zygmund operators}
Given a cube $Q_0\subset {\mathbb R}^d$, let ${\mathcal D}(Q_0)$ denote the set of all dyadic cubes with respect to $Q_0$, that is, the cubes
obtained by repeated subdivision of $Q_0$ and each of its descendants into $2^d$ congruent subcubes.

A dyadic lattice ${\mathscr D}$ in ${\mathbb R}^d$ is any collection of cubes such that
\begin{enumerate}
\renewcommand{\labelenumi}{(\roman{enumi})}
\item
if $Q\in{\mathscr D}$, then each child of $Q$ is in ${\mathscr D}$ as well;
\item
every 2 cubes $Q',Q''\in {\mathscr D}$ have a common ancestor, i.e., there exists $Q\in{\mathscr D}$ such that $Q',Q''\in {\mathcal D}(Q)$;
\item
for every compact set $K\subset {\mathbb R}^d$, there exists a cube $Q\in {\mathscr D}$ containing $K$.
\end{enumerate}

\begin{lemma}\label{3n}
There exist $3^d$ dyadic lattices ${\mathscr D}_j$ such that for every cube $Q\subset{\mathbb R}^d$ there is a cube $R$ from some ${\mathscr D}_j$ which contains $Q$ and $|R|\le 3^d|Q|$.
\end{lemma}

This lemma and the above definition of a dyadic lattice can be found in \cite{LN19}.

Let ${\mathscr D}$ be a dyadic lattice. We say that a family ${\mathcal S}\subset {\mathscr D}$ is $\eta$-sparse, $0<\eta<1$, if for every cube $Q\in {\mathcal S}$,
$$\Big|\bigcup_{Q'\in {\mathcal S}: Q'\subsetneq Q}Q'\Big|\le (1-\eta)|Q|.$$
In particular, if ${\mathcal S}\subset {\mathscr D}$ is $\eta$-sparse, then defining for every $Q\in {\mathcal S}$,
$$E_Q=Q\setminus \bigcup_{Q'\in {\mathcal S}: Q'\subsetneq Q}Q',$$
we obtain that $|E_Q|\ge \eta |Q|$ and the sets $\{E_Q\}_{Q\in {\mathcal S}}$ are pairwise disjoint.
If the sparseness number is nonessential we will skip it by simply saying that a family ${\mathcal S}$ is sparse.

The following result is an immediate combination of \cite[Lemmas 6.3, 6.6]{LN19}.

\begin{lemma}\label{unsparse}
If ${\mathcal S}\subset {\mathscr D}$ is $\eta$-sparse and $m\ge 2$, then one can represent ${\mathcal S}$ as a disjoint union ${\mathcal S}=\cup_{j=1}^m{\mathcal S}_j$, where
each family ${\mathcal S}_j$ is $\frac{m}{m+(1/\eta)-1}$-sparse.
\end{lemma}

The following statement is implicit in \cite{DLR16}. We will give its proof for the sake of completeness.

\begin{lemma}\label{sppr} Let a family ${\mathcal S}\subset {\mathscr D}$ be $\frac{7}{8}$-sparse. For a non-negative locally integrable $\f$ and $\ga>0$, set
$$F:=\{Q\in {\mathcal S}: \ga\le\frac{1}{|Q|}\int_Q\f\le 4\ga\}.$$
Then there exist pairwise disjoint subsets $G_Q\subset Q, Q\in F,$ such that for all $Q\in F$,
$$\int_Q\f\le 8\int_{G_Q}\f.$$
\end{lemma}

\begin{proof}
Define
$$G_Q:=Q\setminus \bigcup_{Q'\in F: Q'\subsetneq Q}Q'.$$
Then the sets $\{G_Q\}_{Q\in F}$ are pairwise disjoint. Next, let $P_j$ be the maximal cubes of the family $\{Q'\in F: Q'\subsetneq Q\}$.
Then they are pairwise disjoint, and
$$\sum_j|P_j|=|Q\setminus G_Q|\le \frac{1}{8}|Q|.$$
Hence,
\begin{eqnarray*}
\ga|Q|&\le&\int_Q\f=\int_{G_Q}\f+\sum_{j}\int_{P_j}\f\\
&\le& \int_{G_Q}\f+4\ga\sum_{j}|P_j|\le \int_{G_Q}\f+\frac{\ga}{2}|Q|.
\end{eqnarray*}
From this, $\ga|Q|\le 2\int_{G_Q}\f$,
and the statement follows from the definition of $F$.
\end{proof}

Given a sparse family ${\mathcal S}$, define for non-negative locally integrable $\f$ the sparse operator $A_{\mathcal S}$ by
$$A_{\mathcal S}(\f)(x):=\sum_{Q\in {\mathcal S}}\Big(\frac{1}{|Q|}\int_Q\f\Big)\chi_Q(x).$$

We say that $T$ is a Dini-continuous Calder\'on--Zygmund operator if $T$ is a linear operator of weak type $(1,1)$ such that
$$Tf(x)=\int_{{\mathbb R}^n}K(x,y)f(y)dy\quad\text{for all}\,\,x\not\in \text{supp}\,f$$
with kernel $K$ satisfying the smoothness condition
$$
|K(x,y)-K(x',y)|\le \o\left(\frac{|x-x'|}{|x-y|}\right)\frac{1}{|x-y|^n}
$$
for $|x-x'|<|x-y|/2$, where $\o$ is a modulus of continuity such that $\int_0^1\o(t)\frac{dt}{t}<\infty.$

The following result is well known, see, e.g.,  \cite{LO20} for a short proof.

\begin{theorem}\label{spd} Let $T$ be a Dini-continuous Calder\'on--Zygmund operator.
There exist $3^d$ dyadic lattices ${\mathscr D}_j$ with the following property:
for every compactly supported and integrable $f$, there exist $\eta_d$-sparse families ${\mathcal S}_j\subset {\mathscr D}_j$ such that
\begin{equation}\label{sd}
|Tf(x)|\lesssim \sum_{j=1}^{3^d}A_{{\mathcal S}_j}(|f|)(x)
\end{equation}
almost everywhere.
\end{theorem}

\subsection{Matrix weights}
Recall first that a scalar weight $w$ satisfies the scalar $A_p,p>1,$ condition if
$$[w]_{A_p}:=\sup_{Q}\Big(\frac{1}{|Q|}\int_Qw\Big)\Big(\frac{1}{|Q|}\int_Qw^{-\frac{1}{p-1}}\Big)^{p-1}<\infty.$$

We will use the following sharp reverse H\"older property from \cite{HP13}.
\begin{prop}\label{rhweights}
There exists a constant $c_{d}>0$ such that for all $w \in A_p$ and every cube $Q\subset {\mathbb R}^d$,
\begin{equation}\label{srh}
\Big(\frac{1}{|Q|}\int_Qw^r\Big)^{1/r}\le 2\frac{1}{|Q|}\int_Qw,
\end{equation}
where $r:=1+\frac{1}{c_d[w]_{A_{p}}}$.
\end{prop}

Turn now to matrix weights. Given a $n\times n$ matrix $W$, define its operator norm by
$$\|W\|:=\sup_{u\in {\mathbb C}^n:|u|=1}|Wu|.$$
Observe that if $\{e_j\}$ is the standard orthogonal basis of ${\mathbb C}^n$, then
\begin{equation}\label{basis}
\|W\|\simeq\sum_{j=1}^n|We_j|.
\end{equation}
Further, if $V,W$ self-adjoint positive definite $n\times n$ matrices, then $\|VW\|=\|WV\|$.

We say that a matrix weight $W\in A_p, p>1,$ if
$$[W]_{A_p}:=\sup_{Q}\frac{1}{|Q|}\int_Q\Big(\frac{1}{|Q|}\int_Q\|W(x)^{1/p}W(y)^{-1/p}\|^{p'}dy\Big)^{p/p'}dx<\infty.$$
This definition was given by Roudenko \cite{R03}. Observe that for the scalar valued weights $w$ we obtain the standard $A_p$-constant $[w]_{A_p}$.
It was also shown in \cite{R03} that if $W\in A_p$, then for any $u\in {\mathbb C}^n$, the scalar weight $|W^{1/p}(x)u|^p$ belongs to the scalar $A_p$ and
\begin{equation}\label{scap}
[|W^{1/p}(x)u|^p]_{A_p}\le [W]_{A_p}.
\end{equation}

We say that $W\in A_1$ if
$$[W]_{A_1}:=\sup_{Q}\esssup_{y\in Q}\frac{1}{|Q|}\int_Q\|W(x)W(y)^{-1}\|dx<\infty.$$

Recall (see \cite[Prop. 1.2]{G03}) that given a norm $\rho$ on ${\mathbb C}^n$, there is a self-adjoint and positive definite matrix $A$, called a reducing operator, such that
$$\rho(u)\simeq |Au|\quad(u\in {\mathbb C}^n).$$
Using this result, given a $p>1$ and a cube $Q\subset {\mathbb R}^d$, one can define a reducing operator $V_{Q,p}$ such that
\begin{equation}\label{vqp}
|V_{Q,p}u|\simeq \Big(\frac{1}{|Q|}\int_Q|W^{-1/p}(y)u|^{p'}dy\Big)^{1/p'}\quad(u\in {\mathbb C}^n).
\end{equation}

For $V_{Q,p}$ such defined we will use the following standard properties.

\begin{prop}\label{pr1} For any $f\in L^p(Q)$,
$$\frac{1}{|Q|}\int_Q|V_{Q,p}^{-1}W^{-1/p}(y)f(y)|dy\lesssim \left(\frac{1}{|Q|}\int_Q|f(y)|^pdy\right)^{1/p}.$$
\end{prop}

\begin{proof} By H\"older's inequality,
$$\frac{1}{|Q|}\int_Q|V_{Q,p}^{-1}W^{-1/p}(y)f(y)|dy\le \left(\frac{1}{|Q|}\int_Q\|V_{Q,p}^{-1}W^{-1/p}(y)\|^{p'}dy\right)^{1/p'}\left(\frac{1}{|Q|}\int_Q|f(y)|^pdy\right)^{1/p}.$$
Next, using (\ref{basis}), (\ref{vqp}) and that the matrices $W^{-1/p}(y)$ and $V_{Q,p}^{-1}$ commute in operator norm, we obtain
\begin{eqnarray*}
\left(\frac{1}{|Q|}\int_Q\|V_{Q,p}^{-1}W^{-1/p}(y)\|^{p'}dy\right)^{1/p'}&\lesssim& \sum_{j=1}^n
\left(\frac{1}{|Q|}\int_Q|W^{-1/p}(y)V_{Q,p}^{-1}e_j|^{p'}dy\right)^{1/p'}\\
&\lesssim& \sum_{j=1}^n|V_{Q,p}V_{Q,p}^{-1}e_j|\le n,
\end{eqnarray*}
which, along with the previous estimate, completes the proof.
\end{proof}

\begin{prop}\label{pr2} Assume that $W\in A_p, p>1$. There is a constant $c_d>0$ such that for $s:=1+\frac{1}{c_d[W]_{A_p}}$,
$$
\Big(\frac{1}{|Q|}\int_Q\|W^{1/p}(x)V_{Q,p}\|^{sp}dx\Big)^{1/s}\lesssim [W]_{A_p}.
$$
\end{prop}

\begin{proof}
Combining (\ref{srh}) with (\ref{scap}) and (\ref{basis}), we obtain
\begin{equation}\label{rh}
\Big(\frac{1}{|Q|}\int_Q\|W^{1/p}(x)V_{Q,p}\|^{sp}dx\Big)^{1/s}\lesssim \frac{1}{|Q|}\int_Q\|W^{1/p}(x)V_{Q,p}\|^{p}dx.
\end{equation}
Further, by (\ref{basis}) and (\ref{vqp}),
$$\frac{1}{|Q|}\int_Q\|W^{1/p}(x)V_{Q,p}\|^{p}dx\lesssim [W]_{A_p},$$
which completes the proof.
\end{proof}

\section{The Christ--Goldberg maximal operator}
In this section we prove Theorems \ref{swtb} and \ref{as}. The proof of Theorem \ref{swtb}
will be based on several ingredients. The first one is a pointwise sparse bound for the local dyadic Christ--Goldberg maximal operator defined by
$$M^d_{Q,W,p}f(x):=\sup_{R\in {\mathcal D}(Q), R\ni x}\frac{1}{|R|}\int_R|W^{1/p}(x)W^{-1/p}(y)f(y)|dy.$$
In what follows, $V_{Q,p}$ denotes the reducing operator defined by (\ref{vqp}).

\begin{lemma}\label{spb} There exists a sparse family ${\mathcal S}\subset {\mathcal D}(Q)$ such that for every $r>0$ and for a.e. $x\in Q$,
\begin{equation}\label{r>0}
M^d_{Q,W,p}f(x)^r\le  2^r\sum_{R\in {\mathcal S}}\left(\|W^{1/p}(x)V_{R,p}\|\frac{1}{|R|}\int_R|V_{R,p}^{-1}W^{-1/p}(y)f(y)|dy\right)^r\chi_{R}(x).
\end{equation}
\end{lemma}

\begin{proof} We will use almost the same argument as in \cite[Section 5.1.2]{IPRR21}. Let $M^d_Q$ denote the standard local dyadic maximal operator in the scalar setting, that is,
$$M^d_Q\f(x)=\sup_{R\ni x, R\in {\mathcal D}(Q)}\frac{1}{|R|}\int_{R}|\f|.$$
Consider the set
$$\Omega:=\{x\in Q:M_Q^d(|V_{Q,p}^{-1}W^{-1/p}f|)>2\frac{1}{|Q|}\int_Q|V_{Q,p}^{-1}W^{-1/p}f|\}.$$
Then $|\Omega|\le \frac{1}{2}|Q|$. Write $\Omega$ as the union of the maximal pairwise disjoint dyadic cubes, $\Omega=\cup_jR_j$.

Assume that $x\in Q$ and $R\in {\mathcal D}(Q)$ is such that $x\in R$ and $R\cap (Q\setminus \Omega)\not=\emptyset$. Then
\begin{eqnarray*}
\frac{1}{|R|}\int_R|W^{1/p}(x)W^{-1/p}(y)f(y)|dy&\le& \|W^{1/p}(x)V_{Q,p}\|\frac{1}{|R|}\int_R|V_{Q,p}^{-1}W^{-1/p}(y)f(y)|dy\\
&\le& 2\|W^{1/p}(x)V_{Q,p}\|\frac{1}{|Q|}\int_Q|V_{Q,p}^{-1}W^{-1/p}(y)f(y)|dy.
\end{eqnarray*}
From this, setting
$$F_j:=\{x\in R_j: M^d_{Q,W,p}f(x)\not=M^d_{R_j,W,p}f(x)\},$$
we obtain that for $x\in \cup_jF_j\cup(Q\setminus \Omega)$,
$$
M^d_{Q,W,p}f(x)\le 2\|W^{1/p}(x)V_{Q,p}\|\frac{1}{|Q|}\int_Q|V_{Q,p}^{-1}W^{-1/p}(y)f(y)|dy.
$$

Hence, for all $x\in Q$,
\begin{eqnarray*}
M^d_{Q,W,p}f(x)&\le& 2\left(\|W^{1/p}(x)V_{Q,p}\|\frac{1}{|Q|}\int_Q|V_{Q,p}^{-1}W^{-1/p}(y)f(y)|dy\right)\chi_{\cup_jF_j\cup(Q\setminus \Omega)}\\
&+&\sum_jM^d_{R_j,W,p}f(x)\chi_{R_j\setminus F_j}.
\end{eqnarray*}
From this, for any $r>0$,
\begin{eqnarray*}
M^d_{Q,W,p}f(x)^r&\le& 2^r\left(\|W^{1/p}(x)V_{Q,p}\|\frac{1}{|Q|}\int_Q|V_{Q,p}^{-1}W^{-1/p}(y)f(y)|dy\right)^r\chi_Q\\
&+&\sum_jM^d_{R_j,W,p}f(x)^r\chi_{R_j}.
\end{eqnarray*}
Iterating this estimate, we obtain a $\frac{1}{2}$-sparse family ${\mathcal S}\subset {\mathcal D}(Q)$, for which (\ref{r>0}) holds.
\end{proof}

Our second ingredient is a sparse operator defined in the scalar-valued setting for a sparse family ${\mathcal S}\subset {\mathscr D}$ and a sequence of non-negative functions $\la:=\{\la_Q\}_{Q\in {\mathcal S}}$ by
$$T_{\la,{\mathcal S}}\psi(x):=\sum_{Q\in {\mathcal S}}\la_Q(x)\Big(\frac{1}{|Q|}\int_Q\psi\Big)\chi_Q(x).$$

\begin{lemma}\label{T} Suppose that there exist $r>1$ and $A>0$ such that
$$\sup_{Q\in {\mathcal S}}\Big(\frac{1}{|Q|}\int_Q\la_Q(x)^rdx\Big)^{1/r}\le A.$$
Then
$$\|T_{\la,{\mathcal S}}\psi\|_{L^{1,\infty}}\lesssim Ar'\|\psi\|_{L^1}.$$
\end{lemma}

\begin{proof} Let us show first that $T_{\la,{\mathcal S}}$ is bounded on $L^{(2r')'}$ and
\begin{equation}\label{bound}
\|T_{\la,{\mathcal S}}\psi\|_{L^{(2r')'}}\lesssim Ar'\|\psi\|_{L^{(2r')'}}.
\end{equation}
By H\"older's inequality and by sparseness,
\begin{eqnarray*}
\int_{{\mathbb R}^d}(T_{\la,{\mathcal S}}\psi)\f&=&\sum_{Q\in {\mathcal S}}\int_Q\la_Q\f\Big(\frac{1}{|Q|}\int_Q\psi\Big)\le A\sum_{Q\in {\mathcal S}}\Big(\frac{1}{|Q|}\int_Q\f^{r'}\Big)^{1/r'}\Big(\frac{1}{|Q|}\int_Q\psi\Big)|Q|\\
&\lesssim& A\sum_{Q\in {\mathcal S}}\int_{E_Q}(M_{r'}\f)M\psi\lesssim A\int_{{\mathbb R}^d}(M_{r'}\f)M\psi.
\end{eqnarray*}
From this, applying H\"older's inequality again, we obtain
\begin{eqnarray*}
\int_{{\mathbb R}^d}(T_{\la,{\mathcal S}}\psi)\f\lesssim A\|M_{r'}\f\|_{L^{2r'}}\|M\psi\|_{L^{(2r')'}}\lesssim Ar'\|\f\|_{L^{2r'}}\|\psi\|_{L^{(2r')'}},
\end{eqnarray*}
which, by duality, implies (\ref{bound}).

Take $\ga>0$ which will be chosen later on. Using the standard Calder\'on--Zygmund decomposition, write the set $\Omega:=\{x\in {\mathbb R}^d: M^{\mathscr D}\psi(x)>\ga\}$ as the union of the
maximal cubes $Q_j\in {\mathscr D}$, and set
$$b:=\sum_j\Big(\psi-\frac{1}{|Q_j|}\int_{Q_j}\psi\Big)\chi_{Q_j},\quad g:=\psi-b.$$
Then, observing that for $x\not\in \Omega$, $T_{\la,{\mathcal S}}b(x)=0$, and using (\ref{bound}), we obtain
\begin{eqnarray*}
|\{x\in {\mathbb R}^d:T_{\la,{\mathcal S}}\psi(x)>1\}|&\le& |\Omega|+|\{x\not\in \Omega: T_{\la,{\mathcal S}}g(x)>1\}|\\
&\lesssim& \frac{1}{\ga}\|\psi\|_{L^1}+(Ar')^{(2r')'}\ga^{(2r')'-1}\|\psi\|_{L^1}.
\end{eqnarray*}
Optimizing this expression with respect to $\ga$, we obtain
$$|\{x\in {\mathbb R}^d:T_{\la,{\mathcal S}}\psi(x)>1\}|\lesssim Ar'\|\psi\|_{L^1},$$
which completes the proof.
\end{proof}

\begin{proof}[Proof of Theorem \ref{swtb}] We start by showing that
\begin{equation}\label{fp}
\|M_{W,p}f\|_{L^{p,\infty}}\lesssim [W]_{A_p}^{\frac{2}{p}}\|f\|_{L^p}.
\end{equation}
The proof follows some ideas used in \cite{CUIMPRR21}.
By Lemma \ref{3n}, it suffices to prove the theorem for the dyadic maximal operator $M^{\mathscr D}_{W,p}f$ (where the supremum is taken over all cubes $Q\in {\mathscr D}$ containing the point $x$).
In turn, by the standard limiting argument, given a fixed cube $Q\in {\mathscr D}$, it suffices to prove the theorem for $M^d_{Q,W,p}f$.

Applying Lemma \ref{spb} with $r=p$ along with Proposition \ref{pr1} yields
$$M^d_{Q,W,p}f(x)^p\lesssim \sum_{R\in {\mathcal S}}\left(\|W^{1/p}(x)V_{R,p}\|^p\frac{1}{|R|}\int_R|f(y)|^pdy\right)\chi_{R}(x).$$
From this, setting
$$\la_R(x):=\|W^{1/p}(x)V_{R,p}\|^p,$$
we obtain that
$$\|M^d_{Q,W,p}\|_{L^p\to L^{p,\infty}}\lesssim \|T_{\la, {\mathcal S}}\|_{L^1\to L^{1,\infty}}^{1/p}.$$
Now, Lemma \ref{T} along with Proposition \ref{pr2} implies
$$\|T_{\la, {\mathcal S}}\|_{L^1\to L^{1,\infty}}\lesssim [W]_{A_p}^2,$$
which, along with the previous estimate, proves (\ref{fp}).

Turn to the second part of the theorem by showing the sharpness of (\ref{fp}) in the scalar case. We will show that for every natural $N\ge 100$, there is a weight $w$ such that $[w]_{A_p}\simeq N$ and
$\|M_{w,p}\|_{L^p\to L^{p,\infty}}\gtrsim N^{2/p}$. From this we will clearly obtain that $\f_{M_p}(t)\gtrsim t^{2/p}$.

The construction of the example is a modification of that in \cite{LLOR23}. We begin with the notations presented there. For $k=3,\ldots, N$ we denote $J_k=[2^k, 2^{k+1})$. We will split $J_k$ into small intervals.
Set $I_k=[2^k, 2^k+k)$ and $L_k= J_k\setminus I_k=[2^k+k, 2^{k+1})$. Let $L_k^{-}$ and $L_k^{+}$ be the left and right halves of $L_k$, respectively. Next we define $(L_k^{-})^1$ to be the right half of $L_k^{-}$ and $(L_k^{+})^1$ the left half of $L_k^{+}$. Then
\begin{enumerate}
\item[$\bullet$]when $(L_k^{-})^j=[a_{k}^j, b_k^j)$ is defined, let $(L_k^{-})^{j+1}=[a_{k}^{j+1}, b_{k}^{j+1})$ satisfy that
\[
b_{k}^{j+1}=a_k^j,\qquad |(L_k^{-})^{j+1}|= \frac 12 |(L_k^{-})^j|;
\]
\item[$\bullet$]when $(L_k^{+})^j=[c_{k}^j, d_k^j)$ is defined, let $(L_k^{+})^{j+1}=[c_{k}^{j+1}, d_{k}^{j+1})$ satisfy that
\[
c_{k}^{j+1}=d_k^j,\qquad |(L_k^{+})^{j+1}|= \frac 12 |(L_k^{+})^j|.
\]
\end{enumerate}
The process is stopped when we have $(L_k^{-})^{k-1}$ and $(L_k^{+})^{k-1}$ defined, and we simply define
\[
(L_k^{-})^{k}= [2^k+k, 2^k+k+ \frac{|L_k^-|}{2^{k-1}}),\,\, (L_k^{+})^{k}= [2^{k+1}-\frac{|L_k^+|}{2^{k-1}}, 2^{k+1}).
\]Now we have
\[
J_k=I_k  \cup   \mathop{\cup}_{j=1}^k \big( (L_k^{-})^j \cup  (L_k^{+})^j \big)=: I_k  \cup   \mathop{\cup}_{j=1}^k L_k^j.
\]

Define
\[
w_k:= 2^{k+1}\chi_{I_k}+ \sum_{j=\lfloor \log_{2} k \rfloor}^{k} 2^{j}\chi_{L_k^j}+ k \sum_{j=1}^{\lfloor \log_{2} k \rfloor-1}  \chi_{L_k^j}.
\]
and our weight on $[0, 2^{N+2}]$ is
\begin{equation*}
w(x):=\begin{cases}
 \chi_{[0, 8)}(x)+ \sum\limits_{k=3}^N 2^{(k+1)(p-1)}w_k(x), \quad & x\in [0, 2^{N+1}),\\
2^{(N+1)p-1}, \quad & x=2^{N+1},\\
w(2^{N+2}-x), \quad & x\in [2^{N+1}, 2^{N+2}].
 \end{cases}
\end{equation*}
Finally we extend $w(x)$ from $[0, 2^{N+2}]$ to $\mathbb R$ periodically with period $2^{N+2}$. Since $w^{1/p}(x)>x$ on $I_k$, we have
\begin{align*}
\| w^{1/p}M(\chi_{[0,1]}) \|_{L^{p,\infty}}^p&\ge |\{x\in (1, \infty): w^{1/p}(x) M(\chi_{[0,1]}) > 1\}|\\
&\ge |\{x\in (1, \infty): w^{1/p}(x)> x\}|\ge \sum_{k=3}^N |I_k|=\sum_{k=3}^Nk\simeq N^2.
\end{align*}
From this, $\|M_{w,p}\|_{L^p\to L^{p,\infty}}\gtrsim N^{2/p}$. Hence, in order to prove the claim, it remains to check that $[w]_{A_p}\simeq N$.
Since $w$ is periodic on $\mathbb R$ and symmetrical on $[0, 2^{N+2}]$, it suffices to prove that
\[
\sup_{I \subset [0, 2^{N+1}]}\frac {w(I)}{|I| } \Big(\frac 1{|I|} \int_I w^{-\frac 1{p-1}}\Big)^{p-1}\simeq N
\]

First, observe that
\[
|L_k^j|=2^{-j}(2^k-k),\quad j=1,2,\ldots, k-1,\qquad |L_k^k|=2^{1-k}(2^k-k).
\]
Hence when $I= [0, 2^{N+1}]$ we have
\begin{align*}
\frac {w(I)}{|I| }&= 2^{-(N+1)}\Big(8+\sum_{k=3}^N 2^{(k+1) (p-1)} \Big(k 2^{k+1}+ \sum_{j=\lfloor \log_{2} k \rfloor}^{k} 2^{j}|L_k^j| + k \sum_{j=1}^{\lfloor \log_{2} k \rfloor-1}|L_k^j|\Big)  \Big)\\
&\simeq 2^{-(N+1)}\Big(8+\sum_{k=3}^N 2^{(k+1) (p-1)} k 2^{k+1}\Big)\simeq N2^{N(p-1)}
\end{align*}
and since $1<p<2$,
\begin{align*}
&\frac 1{|I|} \int_I w^{-\frac 1{p-1}}\\
&=  2^{-(N+1)} \Big(8+ \sum_{k=3}^N 2^{-(k+1)} \Big(k 2^{-\frac{k+1}{p-1}}+ \sum_{j=\lfloor \log_{2} k \rfloor}^{k} 2^{-\frac j{p-1}}|L_k^j| + k^{-\frac 1{p-1}} \sum_{j=1}^{\lfloor \log_{2} k \rfloor-1}|L_k^j| \Big)\Big)\\
&\simeq  2^{-(N+1)} \Big(8+ \sum_{k=3}^N 2^{-(k+1)} \Big(k 2^{-\frac{k+1}{p-1}}+ \sum_{j=\lfloor \log_{2} k \rfloor}^{k} 2^{-\frac j{p-1}} 2^{k-j} + k^{-\frac 1{p-1}} \sum_{j=1}^{\lfloor \log_{2} k \rfloor-1} 2^{k-j}\Big)\Big) \\
&\lesssim  2^{-(N+1)} \Big(8+ \sum_{k=3}^N (k 2^{-kp'}+ k^{-p'}+ k^{-\frac 1{p-1}})\Big)  \\
&\simeq 2^{-N}.
\end{align*}
Thus
\[
\frac {w(I)}{|I| } \Big(\frac 1{|I|} \int_I w^{-\frac 1{p-1}}\Big)^{p-1}\simeq N.
\]

We are left to prove that for any $I\subset [0, 2^{N+1}]$, there holds that
\begin{equation}\label{eq:target}
\frac {w(I)}{|I| } \Big(\frac 1{|I|} \int_I w^{-\frac 1{p-1}}\Big)^{p-1}\lesssim N.
\end{equation}
As that in \cite{LLOR23}, our weight is a step function (one may view $\bigcup_{j=1}^{\lfloor \log_{2} k \rfloor-1}L_k^j$ as a single interval but it is not necessary), and the jump of $w$ from the interval $[0,8)$ to $I_3$ is $2^{4p}$ for each two other adjacent intervals from the definition of $w$, the jump is at most $2^{p+1}$.
From this observation we have the following

\vskip 1mm
{\bf Claim A.} If $I$ intersects at most $m$ intervals from the definition of $w$, then
$$
\frac{w(I)}{|I|}\le \max_{x\in I} w(x)\le 2^{4p+(m-2)(p+1)} \essinf\limits_{x\in I}w(x).
$$
In what follows we will prove \eqref{eq:target} according to the size of $I$.

\vskip 1mm
{\bf Case 1}. $|I|\le 8$. In this case, note that in each $J_k$ $(k\ge 3)$,  $(L_k^-)^{k-1}, (L_k^-)^k$ and $(L_k^+)^{k-1},(L_k^+)^k$ are the smallest intervals, and
\[
|(L_k^-)^{k-1}| =|(L_k^-)^k| =|(L_k^+)^{k-1}|  =  |(L_k^+)^k| = 1-\frac k{2^k}> \frac 12.
\]
Hence $I$ intersects at most $17$ intervals from the definition of $w$, and we are in position to apply Claim A with $m=17$.

\vskip 1mm
{\bf Case 2}. $|I|> 8$. In this case, we may assume $|I|\in (2^{k_0}, 2^{k_0+1}]$ with some $k_0\ge 3$. We may further assume $k_0<N-10$ as otherwise
\begin{align*}
\frac {w(I)}{|I| } \Big(\frac 1{|I|} \int_I w^{-\frac 1{p-1}}\Big)^{p-1}\lesssim \frac {w([0, 2^{N+1}])}{|[0, 2^{N+1}]| } \Big(\frac 1{|[0, 2^{N+1}]|} \int_0^{2^{N+1}} w^{-\frac 1{p-1}}\Big)^{p-1}\simeq N.
\end{align*}
\vskip 1mm
{\bf Case 2a}. $I\subset [0, 2^{k_0+10}]$. Then similarly to above,
\[
\frac {w(I)}{|I| } \Big(\frac 1{|I|} \int_I w^{-\frac 1{p-1}}\Big)^{p-1}\lesssim \frac {w([0, 2^{k_0+10}])}{|[0, 2^{k_0+10}]| } \Big(\frac 1{|[0, 2^{k_0+10}]|} \int_0^{2^{k_0+10}} w^{-\frac 1{p-1}}\Big)^{p-1}\simeq k_0.
\]

\vskip 1mm
{\bf Case 2b}. $I\not\subset [0, 2^{k_0+10}]$. Then $I\subset [2^{k_0+9}, 2^{N+1}]$. Denote by $c_k$ the center of $L_k$.
\vskip 1mm
{\bf Case 2b-a}. $I$ contains some $c_k$ with $k\ge k_0+9$. Then the estimate is trivial since $I\subset (L_k^-)^1\cup (L_k^+)^1$ and $w$ is a constant on $I$.

\vskip 1mm
{\bf Case 2b-b}. $I$ does not contain any $c_k$. In this case we may assume $I\subset (c_{\ell}, c_{\ell+1})$ for some $k_0+8 \le \ell \le N$. Suppose that $I=[a,b]$ and  $a\in (L_\ell^+)^{j}$ for some $j$. If $j \le  \ell- k_0-4$, then
\[
| (L_\ell^+)^{j+1}|= |L_\ell^+| 2^{-(j+1)}= \frac{2^\ell-\ell}{ 2^{j+2}}> 2^{k_0+1},
\]
so that $I$ will intersect at most $(L_\ell^+)^{j}$ and $(L_\ell^+)^{j+1}$ and we again apply Claim A with $m=2$.

If $j \ge \ell- k_0-3$, note that then
\[
I \subset \mathop{\cup}_{j=\ell-k_0-3}^\ell (L_\ell^+)^j \cup \mathop{\cup}_{i=\ell-k_0-2}^{\ell+1} (L_{\ell+1}^-)^i \cup I_{\ell+1}.
\]
Here $i\ge \ell-k_0-2$ since
\[
|(L_{\ell+1}^-)^{\ell-k_0-2} |= \frac{2^{\ell+1}-(\ell+1)}{2^{\ell-k_0-1 }}> \frac{2^{\ell}}{2^{\ell-k_0-1 }}=2^{k_0+1}\ge  \ell(I).
\]
Hence we have
\begin{align*}
&\frac {w(I)}{|I| \essinf\limits_{x\in I} w(x)}\\&\le \frac{\sum\limits_{j=\ell-k_0-3}^\ell w((L_\ell^+)^j )+ \sum\limits_{i=\ell-k_0-2}^{\ell+1} w((L_{\ell+1}^-)^i )+ w(I_{\ell+1}) }{2^{k_0} 2^{\ell-k_0-3} 2^{(\ell+1)(p-1)}}\\
&\lesssim  \frac{\sum\limits_{j=\ell-k_0-3}^\ell \max\{2^j, \ell\}\cdot 2^{\ell-j}+  \sum\limits_{i=\ell-k_0-2}^{\ell+1}\max\{ 2^i, \ell+1\} \cdot 2^{\ell+1-i} + (\ell+1)2^{\ell+2}}{2^\ell }
\lesssim \ell,
\end{align*}
from which \eqref{eq:target} follows immediately.

It remains to consider the case $a\in I_{\ell+1}\cup L_{\ell+1}^-$. However, in this case we just need to discuss whether $b\in (L_{\ell+1}^-)^{j}$ with some  $j \le  \ell- k_0-3$ or not, which is completely similar. This completes the proof.
\end{proof}

\begin{remark}\label{r3} Recall that the Hilbert transform is defined by
$$Hf(x)=\text{P.V.}\int_{{\mathbb R}}\frac{f(t)}{x-t}dt.$$
Let $w$ be the weight constructed in the sharpness part of Theorem \ref{swtb}. Using that  $H(\chi_{[0,1]})(x)>\frac{1}{x}$ for all $x>1$,
we have
$$\|w^{1/p}H(\chi_{[0,1]})\|_{L^{p,\infty}}^p\ge |\{x\in (1, \infty): w^{1/p}(x)> x\}|\gtrsim N^2.$$
From this, exactly as for the maximal operator, we obtain that $\f_{H_{p}}(t)\gtrsim~t^{\frac{2}{p}}$.
\end{remark}

\begin{remark}\label{r1}
Recall that the scalar $A_{\infty}$ constant is defined by
$$[w]_{A_{\infty}}:=\sup_{Q}\frac{1}{w(Q)}\int_QM(w\chi_Q).$$
Given a matrix weight $W$ and $p>1$, define its $A_{\infty, p}^{sc}$ constant by
$$[W]_{A_{\infty, p}^{sc}}:=\sup_{u\in {\mathbb C}^n}[|W^{1/p}u|^p]_{A_{\infty}}.$$
Then it is easy to show that Proposition \ref{pr2} holds with $s=1+\frac{1}{c[W]_{A_{\infty, p}^{sc}}}$.
As a result, a statement of Theorem \ref{swtb} can be written in the form
$$
\|M_{W,p}f\|_{L^{p,\infty}}\lesssim \big([W]_{A_p}[W]_{A_{\infty, p}^{sc}}\big)^{1/p}\|f\|_{L^p}.
$$
\end{remark}

\begin{remark}\label{r2}
Lemma \ref{spb} can also be used in order to give a simple proof of the strong type bound
\begin{equation}\label{sqb}
\|M_{W,p}f\|_{L^p}\lesssim [W]_{A_p}^{\frac{1}{p-1}}\|f\|_{L^p}\quad(1<p<\infty)
\end{equation}
(see \cite{B93,IM19}).

Indeed, as before, it suffices to prove (\ref{sqb}) for $M^d_{Q,W,p}f$. By Lemma~\ref{spb},
\begin{eqnarray*}
\int_QM^d_{Q,W,p}f(x)^pdx&\lesssim& \sum_{R\in {\mathcal S}}\int_R\|W^{1/p}(x)V_{R,p}\|^pdx
\left(\frac{1}{|R|}\int_R|V_{R,p}^{-1}W^{-1/p}(y)f(y)|dy\right)^p\\
&\lesssim& [W]_{A_p}\sum_{R\in {\mathcal S}}\left(\frac{1}{|R|}\int_R|V_{R,p}^{-1}W^{-1/p}(y)f(y)|dy\right)^p|R|\\
&\lesssim& [W]_{A_p}\int_Q\tilde M_{w,Q}^{\mathscr D}f(x)^pdx,
\end{eqnarray*}
where
$$
\tilde M_{w,Q}^{\mathscr D}f(x):=\sup_{R\ni x, R\in {\mathcal D}(Q)}\frac{1}{|R|}\int_R|V_{R,p}^{-1}W^{-1/p}(y)f(y)|dy.
$$
Next, the standard machinery based on the reverse H\"older estimate shows that
$$M_{w,Q}^{\mathscr D}f(x)\lesssim M_{Q,p-\e}(|f|)(x),$$
where $\e\simeq [W]_{A_p}^{-\frac{1}{p-1}}$. Hence,
$$\int_QM^d_{Q,W,p}f(x)^pdx\lesssim [W]_{A_p}\int_QM_{Q,p-\e}(|f|)(x)^pdx\lesssim [W]_{A_p}^{\frac{p}{p-1}}\int_Q|f(x)|^pdx,$$
which proves (\ref{sqb}).
\end{remark}

We conclude this section by proving Theorem \ref{as}.

\begin{proof}[Proof of Theorem \ref{as}]
Denote
$${\mathcal N}:=\|M_{p,w}\|_{L^p\to L^{p,\infty}}.$$
Then trivially
\begin{equation}\label{eq:hardy}
\| w^{\frac 1p} \mathcal H(f w^{-\frac 1p})\|_{L^{p, \infty}}\le \mathcal N \|f\|_{L^p},
\end{equation}
where $\mathcal H$ stands for the Hardy operator
\[
\mathcal H (f)(x):= \frac 1x \int_0^x f(t) d t.
\]
By duality \eqref{eq:hardy} is equivalent to that
\begin{equation}\label{eq:dualhardy}
\|\mathcal H^* (w^{\frac 1p}\chi_E) \|_{L^{p'}(\sigma)}\lesssim \mathcal N |E|^{\frac 1{p'}}
\end{equation}holds for any measurable set $E$, where $\sigma:=w^{-\frac 1{p-1}}$.

Next let us choose specific $w$ and $E$ in \eqref{eq:dualhardy}. The construction of $w$ here is actually a slight modification of the one in the proof of Theorem~\ref{swtb}. For $k=k_0(p)+1, k_0(p)+2,\ldots$, where $k_0(p)$ is the minimal integer such that $k^{p-1}< 2^{k-1}$, let $J_k=[2^k, 2^{k+1})$ and $I_k=[2^k, 2^k+k^{p-1}]$. Then define $(L_k^{-})^j$ and $(L_k^+)^j$ with the same logic, and let $L_k^j=(L_k^{-})^j\cup (L_k^{+})^j$. For sufficiently large $N$, our weight on $[0, 2^{N+2}]$ is
\begin{equation*}
w(x):=\begin{cases}
 \chi_{[0, 2^{k_0(p)+1})}(x)+ \sum\limits_{k=k_0(p)+1}^N 2^{(k+1)(p-1)}w_k(x), \quad & x\in [0, 2^{N+1}),\\
2^{(N+1)p-1}, \quad & x=2^{N+1},\\
w(2^{N+2}-x), \quad & x\in [2^{N+1}, 2^{N+2}],
 \end{cases}
\end{equation*}
where
\[
w_k(x):= 2^{k+1}\chi_{I_k}+ \sum_{j=\lfloor (p-1) \log_{2} k \rfloor}^{k} 2^{j}\chi_{L_k^{j} }+ k^{p-1} \sum_{j=1}^{\lfloor (p-1)\log_{2} k \rfloor-1}  \chi_{L_k^{j} }.
\]

With similar computations as before we get
\[
\sigma(J_k)\gtrsim k 2^{-kp'}+ k^{-p'}+ k^{-1}\ge k^{-1}
\]and
\[
[w]_{A_p}\sim N^{p-1} (\log N)^{p-1}.
\]

Take $E=\cup_{k=k_0(p)+1}^N I_k $. Then for
fixed $k$ and $x\in J_k$, we have
\begin{align*}
{\mathcal H}^* (w^{\frac 1p}\chi_E)(x)&= \int_x^{+\infty}\frac{ w^{\frac 1p}(t)\chi_E(t)}{t}d t\\
&\sim  \int_x^{+\infty} \chi_E (t) d t \sim \sum_{j=k+1}^{N} |I_j|\sim (N^p-k^p).
\end{align*}
In particular, if $k\le N/2$, then
\[
{\mathcal H}^* (w^{\frac 1p}\chi_E)(x) \sim N^p.
\]
Then it follows that
\begin{align*}
\|\mathcal H^* (w^{\frac 1p}\chi_E) \|_{L^{p'}(\sigma)}^{p'} \gtrsim  \sum_{k=k_0(p)+1}^{N/2}  N^{pp'} \sigma(J_k)\gtrsim \sum_{k=k_0(p)+1}^{N/2}  N^{pp'} k^{-1}\sim N^{pp'} \log N.
\end{align*}
Since $|E|\sim N^p$, by \eqref{eq:dualhardy} we have
\[
\mathcal N\gtrsim \frac{N^p (\log N)^{\frac 1{p'}}}{ N^{p-1}}= N  (\log N)^{\frac 1{p'}}\gtrsim \frac{[w]_{A_p}^{\frac 1{p-1}}}{(\log [w]_{A_p})^{\frac 1p}},
\]
from which $\f_{M_{p}}(t)\gtrsim t^{\frac{1}{p-1}} \big(\log (t+e) \big)^{-\frac 1p}.$
\end{proof}

\section{Calder\'on--Zygmund operators}
In this section we prove Theorems \ref{czop} and~\ref{czsh}. Let us start with some preparations needed to prove Theorem \ref{czop}. First, by Theorem \ref{spd}, it suffices
to prove this result for the sparse operator $A_{\mathcal S}$ instead of $T$, where ${\mathcal S}\subset {\mathscr D}$ and ${\mathcal S}$ is $\eta_d$-sparse.
Second, by Lemma~\ref{unsparse}, one can split ${\mathcal S}=\cup_{j=1}^{m_d}{\mathcal S}_j$, where each ${\mathcal S}_j$ will be at least $\frac{7}{8}$-sparse.
Therefore, without loss of generality, we will assume in this section that ${\mathcal S}$ is $\frac{7}{8}$-sparse.

Given a cube $Q\in {\mathcal S}$, denote ${\mathcal S}(Q):=\{Q'\in {\mathcal S}:Q'\subseteq Q\}$. We start with the following weak type estimate for the sparse operator
$A_{{\mathcal S}(Q)}$.

\begin{lemma}\label{wtsop} Let $p>2$. For every measurable set $E\subset Q$,
$$
\|A_{{\mathcal S}(Q)}(w^{-\frac{1}{p}}\chi_E)\|_{L^{p,\infty}(w)}\lesssim [w]_{A_p}^{\frac{2}{p}}\log([w]_{A_p}+e)|E|^{\frac{1}{p}}.
$$
\end{lemma}

\begin{proof} Since the $A_p$-constant is invariant under pointwise multiplication by a constant, it suffices to show that
\begin{equation}\label{its}
w\{x\in Q: A_{{\mathcal S}(Q)}(w^{-\frac{1}{p}}\chi_E)(x)>2\}\lesssim [w]_{A_p}^{2}\log^p([w]_{A_p}+e)|E|.
\end{equation}

Let $M_Q$ denote the maximal operator restricted to a cube $Q$. By the weak type bound of $M$ (see \cite{B93}),
$$
w\{x\in Q:M_Q(w^{-\frac{1}{p}}\chi_E)(x)>1/4\}\lesssim [w]_{A_p}|E|.
$$
Therefore, setting
$$G:=\{x\in Q: A_{{\mathcal S}(Q)}(w^{-\frac{1}{p}}\chi_E)(x)>2, M_Q(w^{-\frac{1}{p}}\chi_E)(x)\le 1/4\},$$
we obtain
\begin{equation}\label{prst}
w\{x\in Q: A_{{\mathcal S}(Q)}(w^{-\frac{1}{p}}\chi_E)(x)>2\}\lesssim [w]_{A_p}|E|+w(G).
\end{equation}

Denote
$$F_k:=\{Q'\in {\mathcal S}(Q):4^{-k-1}<\frac{1}{|Q'|}\int_{Q'}w^{-\frac{1}{p}}\chi_E\le 4^{-k}\}.$$
Then, for $x\in G$,
$$A_{{\mathcal S}(Q)}(w^{-\frac{1}{p}}\chi_E)(x)=\sum_{k=1}^{\infty}A_{F_k}(w^{-\frac{1}{p}}\chi_E)(x).$$
From this, for a natural $N$ which will be chosen later on we have
\begin{eqnarray*}
w(G)&\le& w\{x\in Q: \sum_{k=1}^{N}A_{F_k}(w^{-\frac{1}{p}}\chi_E)(x)>1\}\\
&+&w\{x\in Q: \sum_{k=N+1}^{\infty}A_{F_k}(w^{-\frac{1}{p}}\chi_E)(x)>1\}:=I+II.
\end{eqnarray*}

Let us start by estimating $I$. We have
$$I^{\frac{1}{p}}\le \sum_{k=1}^N\|A_{F_k}(w^{-\frac{1}{p}}\chi_E)\|_{L^p(w)}.$$
By Lemma \ref{sppr}, there exist pairwise disjoint (for $Q'\in F_k$) sets $G_{Q'}\subset Q'$ such that
$$A_{F_k}(w^{-\frac{1}{p}}\chi_E)(x)\le 8\sum_{Q'\in F_k}\Big(\frac{1}{|Q'|}\int_{G_{Q'}}(w^{-\frac{1}{p}}\chi_E)\Big)\chi_{Q'}(x).$$
Hence, by H\"older's inequality for weak norms,
\begin{eqnarray*}
&&\int_QA_{F_k}(w^{-\frac{1}{p}}\chi_E)g\le 8\int_{Q}\Big(\sum_{Q'\in F_k}\Big(\frac{1}{|Q'|}\int_{Q'}g\Big)\chi_{G_{Q'}}\Big)w^{-\frac{1}{p}}\chi_E\\
&&\le 8\int_Q(M_Qg)w^{-\frac{1}{p}}\chi_E\le 8\|(M_Qg)w^{-\frac{1}{p}}\|_{L^{p',\infty}}|E|^{\frac{1}{p}}.
\end{eqnarray*}

By Theorem \ref{swtb} (note that $p'<2$),
$$\|(M_Qg)w^{-\frac{1}{p}}\|_{L^{p',\infty}}\lesssim [w^{-p'/p}]_{A_{p'}}^{\frac{2}{p'}}\|g\|_{L^{p'}(w^{-p'/p})}=[w]_{A_p}^{\frac{2}{p}}\|g\|_{L^{p'}(w^{-p'/p})}.$$
Combining this with the previous estimate, we obtain by duality that
$$\|A_{F_k}(w^{-\frac{1}{p}}\chi_E)\|_{L^{p}(w)}\lesssim [w]_{A_p}^{\frac{2}{p}}|E|^{\frac{1}{p}},$$
which implies
\begin{equation}\label{estI}
I\lesssim N^p[w]_{A_p}^{2}|E|.
\end{equation}

Turn to $II$. Since $\sum_{k=N+1}^{\infty}2^{-k}<1$, we have
\begin{eqnarray*}
II&\le& w\{x\in Q: \sum_{k=N+1}^{\infty}A_{F_k}(w^{-\frac{1}{p}}\chi_E)(x)>\sum_{k=N+1}^{\infty}2^{-k}\}\\
&\le& \sum_{k=N+1}^{\infty}w\{x\in Q: A_{F_k}(w^{-\frac{1}{p}}\chi_E)(x)>2^{-k}\}\\
&\le& \sum_{k=N+1}^{\infty}w\{x\in Q: 4^{-k}\sum_{Q'\in F_k}\chi_{Q'}(x)>2^{-k}\}\\
&\le& \sum_{k=N+1}^{\infty}\sum_{R\in F_k^m}w\{x\in R: \sum_{Q'\in F_k(R)}\chi_{Q'}(x)>2^k\},
\end{eqnarray*}
where $F_k^m$ stands for the maximal cubes of $F_k$ (hence they are pairwise disjoint).

By sparseness, there is an absolute $c>0$ such that
$$|\{x\in R: \sum_{Q'\in F_k(R)}\chi_{Q'}(x)>2^k\}|\le e^{-c2^k}.$$
Combining this with the sharp quantitative reverse H\"older property expressed in (\ref{srh}), we obtain
$$w\{x\in R: \sum_{Q'\in F_k(R)}\chi_{Q'}(x)>2^k\}\le e^{-\frac{c2^k}{[w]_{A_p}}}w(R).$$
It follows that
$$II\le \sum_{k=N+1}^{\infty}4^{kp}e^{-\frac{c2^k}{[w]_{A_p}}}\sum_{R\in F_k^m}4^{-kp}w(R).$$

By the disjointness of $R$ and H\"older's inequality,
\begin{eqnarray*}
&&\sum_{R\in F_k^m}4^{-kp}w(R)\le 4^p\sum_{R\in F_k^m}\Big(\frac{1}{|R|}\int_Rw^{-\frac{1}{p}}\chi_E\Big)^pw(R)\\
&&\le 4^p\sum_{R\in F_k^m}\Big(\frac{1}{|R|}\int_Rw^{-p'/p}\Big)^{p/p'}\frac{w(R)}{|R|}|E\cap R|\\
&&\le 4^p[w]_{A_p}\sum_{R\in F_k^m}|E\cap R|\le 4^p[w]_{A_p}|E|.
\end{eqnarray*}
Thus, we obtain
$$II\lesssim [w]_{A_p}|E|\sum_{k=N+1}^{\infty}4^{kp}e^{-\frac{c2^k}{[w]_{A_p}}}.$$

Now observe that one can choose $N\simeq \log([w]_{A_p}+e)$ so that
$$4^{kp}e^{-\frac{c2^k}{[w]_{A_p}}}\le 2^{-k}$$
for every $k\ge N$. Then we obtain that
$$II\lesssim [w]_{A_p}|E|.$$

Combining this estimate with (\ref{estI}) yields
$$w(G)\lesssim [w]_{A_p}^2\log^p([w]_{A_p}+e)|E|.$$
This along with (\ref{prst}) proves (\ref{its}), and therefore the proof is complete.
\end{proof}

\begin{remark}\label{glv} Observe that Lemma \ref{wtsop} implies easily its global version with $Q={\mathbb R}^d$ and an arbitrary sparse family ${\mathcal S}$. Indeed, by the limiting argument,
one can assume that ${\mathcal S}$ is finite. Then one can write ${\mathcal S}=\cup_j{\mathcal S}(Q_j)$, where $Q_j$ are the maximal cubes of ${\mathcal S}$, and apply (\ref{its}) for each $Q_j$.
\end{remark}

An important ingredient in the proof of Theorem \ref{czop} is the following
equivalence relation of Cascante--Ortega--Verbitsky \cite{COV04} saying that for every dyadic lattice ${\mathscr D}$ and a non-negative sequence
$\{\la_Q\}_{Q\in {\mathscr D}}$,
\begin{equation}\label{cov}
\Big\|\sum_{Q\in {\mathscr D}}\la_Q\chi_Q\Big\|_{L^p(w)}\simeq \Big(\sum_{Q\in {\mathscr D}}\la_Q\Big(\frac{1}{w(Q)}\sum_{Q'\in {\mathscr D}, Q'\subseteq Q}\la_{Q'}w(Q')\Big)^{p-1}w(Q)\Big)^{\frac{1}{p}}.
\end{equation}

\begin{proof}[Proof of Theorem \ref{czop}]
Let ${\mathcal N}$ denote the best possible constant in the inequality
$$\|w^{\frac{1}{p}}A_{\mathcal S}(fw^{-\frac{1}{p}})\|_{L^{p,\infty}}\le {\mathcal N}\|f\|_{L^p}.$$
By duality this is equivalent to that for every measurable set $E$,
$$\|A_{\mathcal S}(w^{\frac{1}{p}}\chi_E)\|_{L^{p'}(\si)}\lesssim {\mathcal N}|E|^{\frac{1}{p'}},$$
where $\si:=w^{-\frac{1}{p-1}}$.
Further, by (\ref{cov}),
\begin{equation}\label{cov1}
\|A_{\mathcal S}(w^{\frac{1}{p}}\chi_E)\|_{L^{p'}(\si)}\lesssim
\Big(\sum_{Q\in {\mathcal S}}\la_Q\Big(\frac{1}{\si(Q)}\sum_{Q'\in {\mathcal S}(Q)}\la_{Q'}\si(Q')\Big)^{\frac{1}{p-1}}\si(Q)\Big)^{\frac{1}{p'}},
\end{equation}
where $\la_Q:=\frac{1}{|Q|}\int_Q(w^{\frac{1}{p}}\chi_E)dx$.

By H\"older's inequality for weak norms,
\begin{equation}\label{hwn}
\sum_{Q'\in {\mathcal S}(Q)}\la_{Q'}\si(Q')=\int_QA_{\mathcal S(Q)}(w^{\frac{1}{p}}\chi_E)\si\le \|A_{{\mathcal S}(Q)}(w^{\frac{1}{p}}\chi_E)\|_{L^{p',\infty}(\si)}\si(Q)^{\frac{1}{p}}.
\end{equation}
Next, by Lemma \ref{wtsop} (note that $p'>2$),
\begin{eqnarray*}
&&\|A_{{\mathcal S}(Q)}(w^{\frac{1}{p}}\chi_E)\|_{L^{p',\infty}(\si)}=\|A_{{\mathcal S}(Q)}(\si^{-\frac{1}{p'}}\chi_E)\|_{L^{p',\infty}(\si)}\\
&&\lesssim [\si]_{A_{p'}}^{\frac{2}{p'}}\log([\si]_{A_{p'}}+e)|E\cap Q|^{\frac{1}{p'}}\simeq [w]_{A_p}^{\frac{2}{p}}\log([w]_{A_p}+e)|E\cap Q|^{\frac{1}{p'}}.
\end{eqnarray*}

Combining this estimate with (\ref{hwn}) and (\ref{cov1}) yields
\begin{equation}\label{fs}
\|A_{\mathcal S}(w^{\frac{1}{p}}\chi_E)\|_{L^{p'}(\si)}\lesssim C([w]_{A_p})\Big(\sum_{Q\in {\mathcal S}}
\Big(\frac{1}{|Q|}\int_Q(w^{\frac{1}{p}}\chi_E)\Big)\si(Q)^{\frac{1}{p'}}|E\cap Q|^{\frac{1}{p}}\Big)^{\frac{1}{p'}},
\end{equation}
where
$$C([w]_{A_p}):=[w]_{A_p}^{\frac{2}{p^2}}\log^{\frac{1}{p}}([w]_{A_p}+e).$$

Let $r:=1+\frac{1}{c_d[w]_{A_p}}$ as in Proposition \ref{rhweights}. Then
\begin{eqnarray*}
\frac{1}{|Q|}\int_Q(w^{\frac{1}{p}}\chi_E)&\le& \Big(\frac{1}{|Q|}\int_Qw^r\Big)^{\frac{1}{pr}}\left(\frac{|E\cap Q|}{|Q|}\right)^{\frac{1}{(pr)'}} \\
&\le& 2^{1/p}\Big(\frac{1}{|Q|}\int_Qw\Big)^{\frac{1}{p}}\left(\frac{|E\cap Q|}{|Q|}\right)^{\frac{1}{(pr)'}}
\end{eqnarray*}
Therefore,
$$
\sum_{Q\in {\mathcal S}}
\Big(\frac{1}{|Q|}\int_Q(w^{\frac{1}{p}}\chi_E)\Big)\si(Q)^{\frac{1}{p'}}|E\cap Q|^{\frac{1}{p}}\lesssim [w]_{A_p}^{\frac{1}{p}}\sum_{Q\in {\mathcal S}}
\left(\frac{|E\cap Q|}{|Q|}\right)^{\frac{1}{p}+\frac{1}{(pr)'}}|Q|.
$$

By sparseness, take pairwise disjoint sets $E_Q\subset Q$ such that $|E_Q|\simeq |Q|$. Then
\begin{eqnarray*}
&&\sum_{Q\in {\mathcal S}}
\left(\frac{|E\cap Q|}{|Q|}\right)^{\frac{1}{p}+\frac{1}{(pr)'}}|Q|\lesssim \sum_{Q\in {\mathcal S}}\int_{E_Q}(M\chi_E)^{1+\frac{1}{p}(1-\frac{1}{r})}dx\\
&&\le \int_{{\mathbb R}^d}(M\chi_E)^{1+\frac{1}{p}(1-\frac{1}{r})}\lesssim \frac{1}{r-1}|E|\lesssim [w]_{A_p}|E|.
\end{eqnarray*}

Combining the two previous estimates with (\ref{fs}), we obtain
$$\|A_{\mathcal S}(w^{\frac{1}{p}}\chi_E)\|_{L^{p'}(\si)}\lesssim C([w]_{A_p})[w]_{A_p}^{(1+\frac{1}{p})\frac{1}{p'}}|E|^{\frac{1}{p'}}.$$
Therefore,
$${\mathcal N}\lesssim C([w]_{A_p})[w]_{A_p}^{(1+\frac{1}{p})\frac{1}{p'}}=[w]_{A_p}^{1+\frac{1}{p^2}}\log^{\frac{1}{p}}([w]_{A_p}+e),$$
which completes the proof.
\end{proof}

Turn to the proof of Theorem \ref{czsh}. As we mentioned in the Introduction, the example used here is much simpler than in the previous proofs.

\begin{proof}[Proof of Theorem \ref{czsh}] Denote
$${\mathcal N}:=\|H_{p,w}\|_{L^p\to L^{p,\infty}}.$$
Then, by duality,
$$\|w^{-\frac{1}{p}}H(w^{\frac{1}{p}}\chi_E)\|_{L^{p'}}\lesssim {\mathcal N}|E|^{\frac{1}{p'}}\quad(p\ge 2).$$
Taking here $E=(0,1)$, we obtain
\begin{equation}\label{yef}
\Big(\int_0^1w^{\frac{1}{p}}dx\Big)\Big(\int_2^{\infty}\frac{w^{-\frac{1}{p-1}}(x)}{x^{p'}}dx\Big)^{\frac{1}{p'}}\lesssim {\mathcal N}.
\end{equation}

Let $w_{\e}$ be a radial weight on ${\mathbb R}$ defined on $[0,\infty)$ by
$$w_{\e}(x):=
\begin{cases} \e^{-1},&0\le x\le 1,\\
x^{-(1-\e)},&x>1.
\end{cases}
$$
Then the left-hand side of (\ref{yef}) is equivalent to $\e^{-1}$. Therefore, the lower bound $\f_{H_{p}}(t)\gtrsim t$ would follow
if we show that $[w_{\e}]_{A_p}\lesssim \e^{-1}.$

Denote
$$A_p(w_{\e};I):=\Big(\frac{1}{|I|}\int_Iw_{\e}\Big)\Big(\frac{1}{|I|}\int_Iw_{\e}^{-\frac{1}{p-1}}\Big)^{p-1}.$$
Since $w_{\e}$ is radial, it suffices to consider the intervals $I\subset [0,\infty)$.

Denote $v_{\e}(x):=|x|^{-(1-\e)}$. We will use the well known fact that $[v_{\e}]_{A_p}\simeq \e^{-1}.$ Hence, the case where $I\subset (1,\infty)$ is trivial.
Suppose that $I\cap[0,1]\not=\emptyset$. Then we have to consider only two cases. Assume that $|I|\le 1$. Then
$$A_p(w_{\e};I)\le \frac{\sup_Iw_{\e}}{\inf_Iw_{\e}}\lesssim \e^{-1}.$$
Suppose now that $|I|>1$. Then
$$A_p(w_{\e};I)\lesssim \sup_{h>2}A_p(w_{\e};(0,h)).$$
Observe that for $h>2$,
$$\frac{1}{h}\int_0^hw_{\e}=\frac{1}{\e}h^{\e-1}$$
and
$$\Big(\frac{1}{h}\int_0^hw_{\e}^{-\frac{1}{p-1}}\Big)^{p-1}\lesssim \Big(\frac{1}{h}\big(\e^{\frac{1}{p-1}}+h^{\frac{p-\e}{p-1}}\big)\Big)^{p-1}\lesssim h^{1-\e}.$$
Therefore,
$$\sup_{h>1}A_p(w_{\e};(0,h))\lesssim \frac{1}{\e},$$
and the proof is complete.
\end{proof}


\begin{thebibliography}{99}
\bibitem{B93}
S.M. Buckley, {\it Estimates for operator norms on weighted spaces
and reverse Jensen inequalities}, Trans. Amer. Math. Soc., {\bf 340}
(1993), no. 1, 253--272.

\bibitem{COV04}
C. Cascante, J.M. Ortega, and I.E. Verbitsky, {\it Nonlinear potentials and two
weight trace inequalities for general dyadic and radial kernels}, Indiana Univ.
Math. J. {\bf 53} (2004), no. 3, 845--882.

\bibitem{CG01}
M. Christ and M. Goldberg, {\it Vector $A_2$ weights and a Hardy--Littlewood maximal function},
Trans. Amer. Math. Soc. {\bf 353} (2001), no. 5, 1995--2002.

\bibitem{CIM18}
D. Cruz-Uribe, J. Isralowitz and K. Moen, {\it Two weight bump conditions for matrix weights},
Integr. Equ. Oper. Theory {\bf 90} (2018), no. 3. Paper No. 36, 31 pp.

\bibitem{CUIMPRR21}
D. Cruz-Uribe, J. Isralowitz, K. Moen, S. Pott and I.P. Rivera-R\'ios, {\it Weak endpoint bounds for matrix weights}. Rev. Mat. Iberoam. {\bf 37} (2021), no. 4, 1513--1538.

\bibitem{CUMP05}
D. Cruz-Uribe, J.M. Martell and C. P\'erez, {\it Weighted weak-type inequalities and a conjecture of Sawyer}, Int. Math. Res. Not. 2005, no. 30, 1849--1871.

\bibitem{CUS23}
D. Cruz-Uribe and B. Sweeting, {\it Weighted weak-type inequalities for maximal operators and singular integrals}, preprint.
Available at https://arxiv.org/abs/2311.00828

\bibitem{DLR16}
C. Domingo-Salazar, M.T. Lacey and  G. Rey, {\it Borderline weak type estimates for singular integrals and square functions},
Bull. Lond. Math. Soc., {\bf 48} (2016), no. 1, 63--73.

\bibitem{G03}
M. Goldberg, {\it Matrix $A_p$ weights via maximal functions}, Pacific J. Math. {\bf 211} (2003), no. 2, 201--220.

\bibitem{IM19}
J. Isralowitz and K. Moen, {\it Matrix weighted Poincar\'e inequalities and applications to degenerate elliptic systems}, Indiana Univ. Math. J. {\bf 68} (2019), no. 5, 1327--1377.

\bibitem{IPRR21}
J. Isralowitz, S. Pott and I.P. Rivera-R\'ios, {\it Sharp $A_1$ weighted estimates for vector valued operators}, J. Geom. Anal. {\bf 31} (2021), no. 3, 3085--3116.

\bibitem{H12}
T.P. Hyt\"onen, {\it The sharp weighted bound for general Calder\'on--Zygmund operators}, Ann. of Math. (2) {\bf 175} (2012), no. 3, 1473--1506.

\bibitem{HP13}
T.P. Hyt\"onen and C. P\'erez, {\it Sharp weighted bounds involving $A_{\infty}$}, Anal. PDE {\bf 6} (2013), no. 4, 777--818.

\bibitem{LLOR23}
A.K. Lerner, K. Li, S. Ombrosi and  I.P. Rivera-R\'ios, {\it On the sharpness of some matrix weighted endpoint estimates}, preprint.
Available at https://arxiv.org/abs/2310.06718

\bibitem{LN19}
A.K. Lerner and F. Nazarov, {\it Intuitive dyadic calculus: the basics}, Expo. Math. {\bf 37} (2019), no. 3, 225--265.

\bibitem{LO20}
A.K. Lerner and S. Ombrosi, {\it Some remarks on the pointwise sparse domination}, J. Geom. Anal., {\bf 30} (2020), no. 1, 1011--1027.

\bibitem{LOP19}
K. Li, S. Ombrosi and C. P\'erez, {\it Proof of an extension of E. Sawyer's conjecture about weighted mixed weak-type estimates}. Math. Ann. {\bf 374} (2019), no. 1-2, 907--929.

\bibitem{MW77}
B. Muckenhoupt and R.L. Wheeden, {\it Some weighted weak-type inequalities for the Hardy-Littlewood maximal function and the Hilbert transform}, Indiana Univ. Math. J. {\bf 26} (1977), no. 5, 801--816.

\bibitem{NPTV17}
F. Nazarov, S. Petermichl, S. Treil and A. Volberg, {\it Convex body domination and weighted estimates with matrix weights}, Adv. Math. {\bf 318} (2017), 279--306.

\bibitem{NSS24}
Z. Nieraeth, C. Stockdale and  B. Sweeting, {\it Weighted weak-type bounds for multilinear singular integrals}, preprint. Available at https://arxiv.org/abs/2401.15725. 

\bibitem{R03}
S. Roudenko, {\it Matrix-weighted Besov spaces}, Trans. Amer. Math. Soc. {\bf 355} (2003), no. 1, 273--314.

\bibitem{S85}
E. Sawyer, {\it A weighted weak type inequality for the maximal function}, Proc. Amer. Math. Soc. {\bf 93} (1985), no. 4, 610--614.

\end{thebibliography}
\end{document}